\documentclass[12pt]{amsart}
\usepackage{amsmath,amssymb,amsthm, tikz, mathtools, enumitem}

\newcommand\Oo{{\mathcal{O}}}
\newcommand\p{{\mathfrak{p}}}

\newcommand\B{{\mathfrak{P}}}

\newcommand\Z{\mathbb{Z}}
\newcommand\Q{\mathbb{Q}}

\newenvironment{Gal}{{\rm Gal\,}}{}

\newenvironment{Aut}{{\rm Aut}}{}

\newtheorem{theorem}{Theorem}[section]

\newtheorem{lemma}[theorem]{Lemma}
\newtheorem{proposition}[theorem]{Proposition}
\newtheorem{proposition-definition}[theorem]{Proposition-Definition}
\newtheorem{corollary}[theorem]{Corollary}

\theoremstyle{definition}

\theoremstyle{remark}

\begin{document}
\title{Wild ramification in a family of low-degree extensions arising from iteration}
\author{Benjamin Breen}
\author{Rafe Jones}
\author{Tommy Occhipinti}
\author{Michelle Yuen}

\begin{abstract}
This article gives a first look at wild ramification in a family of iterated extensions. For $c \in \Z$, we consider the splitting field of $(x^2 + c)^2 + c$, the second iterate of $x^2 + c$. We give complete information on the factorization of the ideal $(2)$ as $c$ varies, and find a surprisingly complicated dependence of this factorization on the parameter $c$. We show that $2$ ramifies (necessarily wildly) in all these extensions except when $c = 0$, and we describe the higher ramification groups in some totally ramified cases.
\end{abstract}

\maketitle

\section{Introduction}

Around a decade ago, the authors of \cite{hajir} posed the problem of studying ramification in extensions of number fields generated by iterated polynomials. Specifically, let $L$ be a number field, let $t_0 \in L$, let $f(x) \in L[x]$ have degree $d \geq 2$, and denote by $f^n(x)$ the $n$th iterate of $f$. The main objects of study in \cite{hajir} are the extensions $L_n(f, t_0)/L$ obtained by taking the splitting field of $f^n(x) - t_0$ over $L$. Ramification in these extensions holds particular interest, and one of the main results of \cite{hajir} is that when $f$ is \textit{post-critically finite}, i.e. each critical point of $f$ has a finite forward orbit under iteration, then the set of primes of $L$ ramifiying in $L_n(f, t_0)$ for at least one $n \geq 1$ is a finite set. However, there is also interest in the non-post-critically finite case: on \cite[p. 858]{hajir}, the authors ask what can be said in general about the presence of wild ramification in the extensions $L_n(f, t_0)/L$. In this article, we study perhaps the simplest non-trivial case of this question: $L = \Q$, $t_0 = 0$, $f(x) = x^2 + c$, $c \in \Z$. (Note that any wild ramification must occur above the prime $(2)$ of $\Z$.) Moreover, we restrict our attention to the case where $n = 2$, the smallest $n$ where the extensions incorporate iteration of $f$. We give a complete classification of the of the factorization of the ideal $(2)$ in the extensions $L_2(f, 0)$, which have degree at most 8 (see Theorem \ref{main1} for the most difficult part of this classification) and compute higher ramification groups in the totally ramified case (see Section \ref{ramification}). Even with these severe restrictions, it is obvious from a glance at Theorem \ref{main1} that the dependence on $c$ of the ideal factorization of $(2)$ is remarkably complicated. Then again, perhaps this complexity is not so surprising in light of the difficulty of fully understanding wild ramification in any naturally occurring family of number fields. For instance, even the class of radical extensions $\Q(\zeta_m, \sqrt[m]{a})$ presents impressive complexities (see for instance \cite{viviani}, and also \cite{obus}, \cite{sharifi1}, \cite{sharifi2} for related work). 

We fix notation that will be in use throughout. Denote by $f_c(x)$ the map $x^2 + c$, where $c \in \Z$. Write $L_c$ for $L_2(f_c, 0)$,  the splitting field of the second iterate 
$$f_c(f_c(x)) = (x^2 + c)^2 + c = x^4 + 2cx^2 + c^2 + c.$$
The degree of $L_c$ over $\Q$ depends on whether $-c$ and $-(c+1)$ are squares; if neither is, then it follows that $[L_c : \Q] = 8$ and $\Gal(L_c/\Q) \cong D_4$, the dihedral group of order 8. When $-c$ or $-(c+1)$ is a square (and $c \not\in \{-1,0\}$), then $[L_c : \Q] = 4$ and $\Gal(L_c/\Q)$ is elementary abelian or cyclic, respectively. See Section \ref{degree} for details. 


One of our main results is the following:
\begin{theorem} \label{main0}
Let $f_c(x) = x^2 + c$ with $c \in \Z \setminus \{0\}$, and let $L_c$ be the splitting field of $f_c^2(x)$. Then $L_c$ is wildly ramified at 2. 
\end{theorem}
Note that one can certainly fail to obtain ramification at 2 in the splitting field of $f_c(x)$; for instance this is the case when $c \equiv 3 \bmod{4}$. 



Because $L_c$ is Galois, we have
\begin{equation} \label{genfact}
2\Oo_{L_c} = \prod_{i=1}^g \p_i^e,
\end{equation}
where $\Oo_{L_c}$ is the ring of integers of $L_c$, each $\p_i$ has a common residue degree $f := [(\Oo_{L_c} /\p_i) : (\Z/2\Z)]$ and $efg = [L_c : \Q]$. 
We give complete information on the factorization of the ideal $2\Oo_{L_c}$ for all $c \not\in \{-1,0\}$. In the generic case when $-c$ and $-(c+1)$ are both non-squares, we have $[L_c : \Q] = 8$ (see Section \ref{degree}). We show the following.
 
\begin{theorem} \label{main1}
Let $c \in \Z$ with $-c$ and $-(c+1)$ not squares in $\Z$. Let $L_c$ be the splitting field of $(x^2 + c)^2 + c$, and consider the factorization given in \eqref{genfact}. We have
\begin{align*}
e = 8, f = 1, g = 1 & \qquad \text{iff one of}  \; \begin{cases} \text{$c \equiv 1 \bmod{4}$} \\ \text{$c = 2^{2k+1}m$, where $k \geq 1$ and $m$ is odd} \end{cases} \\
e = 4, f = 2, g = 1 & \qquad \text{iff one of}  \; \begin{cases} \text{$c \equiv 2 \bmod{8}$} \\ \text{$c \equiv 3 \bmod{16}$}
\\ \text{$c \equiv 4 $ or $12 \bmod{32}$} 
\\ \text{$c = 4^{k}(8r \pm 3)$, where $k \geq 2$ and $r \in \Z$}
\end{cases} \\
e = 4, f = 1, g = 2 & \qquad \text{iff one of}  \; 
\begin{cases}
\text{$c \equiv 6 \bmod{8}$} \\
\text{$c \equiv 23$ or $28 \bmod{32}$} \\
\text{$c \equiv 16 \bmod{128}$} \\
\text{$c = 4^{2k}(8r + 1)$, where $k \geq 2$ and $r \in \Z$} \\
\text{$c = 4^{2k-1}(8r + 7)$, where $k \geq 2$ and $r \in \Z$} \\
\text{$c = -1 + 2^{k}(4r + 1)$, where $k \geq 4$ and $r \in \Z$} \\
\end{cases} \\
e = 2, f = 2, g = 2 & \qquad \text{iff one of}  \;
\begin{cases}
\text{$c \equiv 11 \bmod{16}$} \\
\text{$c \equiv 39$ or $52 \bmod{64}$} \\
\text{$c \equiv 240 \bmod{256}$} \\
\text{$c = -1 + 4^{k}(64r +24)$, where $k \geq 1$ and $r \in \Z$} \\
\text{$c=-1 + 4^k(8r+3)$ with $k \ge 2$ and $r\in \Z$} \\
\text{$c = 4^{2k-1} (16r+9)$, where $k \geq 2$ and $r \in \Z$} \\
\text{$c = 4^{2k}(16r + 7)$, where $k \geq 2$ and $r \in \Z$} \\
\end{cases} \\
e = 2, f = 1, g = 4 & \qquad \text{iff one of}  \;
\begin{cases}
\text{$c \equiv 7$ or $20 \bmod{64}$} \\
\text{$c \equiv 112 \bmod{256}$} \\
\text{$c = -1 + 4^{k}(64r - 8)$, where $k \geq 1$ and $r \in \Z$} \\
\text{$c=-1 + 4^k(8r+7)$, where $k\ge 2$ and $r \in \Z$} \\
\text{$c = 4^{2k-1} (16r+1)$, where $k \geq 2$ and $r \in \Z$} \\
\text{$c = 4^{2k}(16r + 15)$, where $k \geq 2$ and $r \in \Z$} \\
\end{cases}
\end{align*}
and these are the only possibilities.
\end{theorem}
When either $-c$ or $-(c+1)$ is a square in $\Z$, we give corresponding classifications in Propositions \ref{biquad} and \ref{cyclic}. 

From Theorem \ref{main1}, we have that $2\Oo_{L_c}$ is totally ramified if and only if $c \equiv 1 \bmod{4}$ or $c = 2^{2k+1}m$, where $k \geq 1$ and $m$ is odd. Note that in both of these cases, we have that $-c$ and $-(c+1)$ are non-squares, and so Theorem \ref{main1} applies. In these cases, we compute the associated ramification filtration in Section \ref{ramification}. We show:

\begin{theorem} \label{totram}
Let $(G_i)_{i \geq 0}$ be the ramification filtration of $\Gal(L_c/\Q)$.
If $c \equiv 1 \bmod{4}$, we have 
$$\#G_0 = \#G_1 = 8, \quad \#G_2 = \#G_3 = 4, \quad \#G_4 = \#G_5 = 2, \quad \text{$\#G_i = 1$ for $i \geq 6$}.
$$
If $c = 2^{2k+1}m$, where $k \geq 1$ and $m$ is odd, we have 
$$\#G_0 = \#G_1 = 8, \quad \#G_2 = \#G_3 = 4, \quad \text{$\#G_i = 2$ for $4 \leq i \leq 7$}, \quad \text{$\#G_i = 1$ for $i \geq 8$}.
$$
\end{theorem}
See Section \ref{ramification} for definitions and more precise results.

The structure of this article is as follows. In Section \ref{degree} we determine the degree and Galois group of $L_c$ for various values of $c$. In Section \ref{global} we present background material on global methods to determine the factorization of a prime in an extension field in various cases, and use this to determine the factorization of $2\Oo_{L_c}$ in the case where $-c$ is a square in $\Z$ (Proposition \ref{biquad}). In Section \ref{mainglob} we prove some cases of Theorem \ref{main1} using the methods of Section \ref{global}. In Section \ref{local} we give background on local methods for determining the factorization of a primein an extension field, and use this to prove Theorem \ref{main0} and determine the factorization of $2\Oo_{L_c}$ in the case where $-(c+1)$ is a square in $\Z$ (Proposition \ref{cyclic}). In Section \ref{mainloc} we complete the proof of Theorem 1.2, using the methods of Section \ref{local}. In Section \ref{ramification} we prove theorem \ref{totram}.

While it would be possible to present a purely local proof of Theorem \ref{main1} and many of the other results in this article, we find the interplay between the global and local methods instructive as to their particular strengths.



\section{Degree and Galois groups} \label{degree}

In this section we compute the degree $[L_c : \Q]$ and the Galois group $\Gal(L_c/\Q)$. 

\begin{proposition} \label{irred}
The polynomial $f_c^2(x)$ is irreducible over $\Q$ if $-c$ is not a square in $\Z$. When $-c$ is a square in $\Z$ and $c \neq 0, -1$, then upon writing $c = -b^2$ for $b \in \Z$ we have
\begin{equation} \label{factor}
f_c^2(x) = (x^2 - (b^2 - b))(x^2 - (b^2 + b)),
\end{equation}
where each of the quadratic factors is irreducible.
\end{proposition}

\begin{proof}
Because $f_c^2(x)$ is only of degree 4, this can be done via a straightforward elementary argument; the key point is that when $-c$ is not a square in $\Z$, $f_c^2(x)$ is irreducible provided that $c^2 + c$ is not a square in $\Z$, which holds for all $c \neq 0, -1$ and in particular for all $c$ with $-c$ not a square. For a much more general result applying to all iterates of $f_c$, see \cite[Proposition 4.5]{quaddiv}.
\end{proof}

The resolvent cubic of $f_c^2(x)$ is 
\begin{equation} \label{resolvent}
x^3 - 2cx^2 - (4c^2 + 4c)x + 8c^3 + 8c^2 = (x - 2c)(x^2 - 4c^2 - 4c).
\end{equation}
Because this has a root in $\Q$, it follows that $\Gal(L_c/\Q)$ is a subgroup of $D_4$, the dihedral group of order 8 \cite[Section 14.6]{DF}. Another way to see this is to note that one may form a binary rooted tree $T$ with root 0 and vertices consisting of the roots of $f_c(x)$ and $f_c^2(x)$, where vertices $v_1$ and $v_2$ are connected if and only if $f_c(v_1) = v_2$. Then it is easy to see that $\Gal(L_c/\Q)$ injects into $\Aut(T)$, and the latter is isomorphic to $D_4$.

Moreover, note that the roots of the quadratic factor in \eqref{resolvent} are $\pm \sqrt{4(c^2 + c)}$. But $c^2 + c$ is not a square in $\Z$ since $c \not\in \{-1,0\}$, and it follows that the quadratic factor in \eqref{resolvent} is irreducible. It thus follows from \cite[Section 14.6]{DF} that $\Gal(L_c/\Q)$ is isomorphic to either $D_4$ or $\Z/4\Z$. 

Note further that the roots of $f_c^2(x)$ are $\{\pm \alpha, \pm \beta\}$, where 
\begin{equation} \label{ab}
\alpha = \sqrt{-c + \sqrt{-c}} \qquad  \text{and} \qquad \beta = \sqrt{-c - \sqrt{-c}}.
\end{equation}
We have $\alpha \beta = \sqrt{-c}\sqrt{-(c+1)}$, and hence $L_c = \Q(\alpha)(\beta) = \Q(\alpha)(\sqrt{-(c+1)})$.

Suppose now that $-c$ and $-(c+1)$ are both non-squares in $\Z$.   We saw above that when $-c$ is not a square that $\Gal(L_c/\Q)$ must be either isomorphic to the cyclic of order 4 or to $D_4$.  It is evident that, under the addtional hypothesis that $-(c+1)$ is not a square, the field $L_c$ contains at least two distinct quadratic subextensions of $\Q$: $\Q(\sqrt{-c})$ and $\Q{\sqrt{-(c+1)}}$, and thus $\Gal(L_c/\Q)$ cannot be cyclic.  
Hence $\Gal(L_c/\Q) \cong D_4$. It is now straightforward to write down its subfield lattice; see Figure 1.

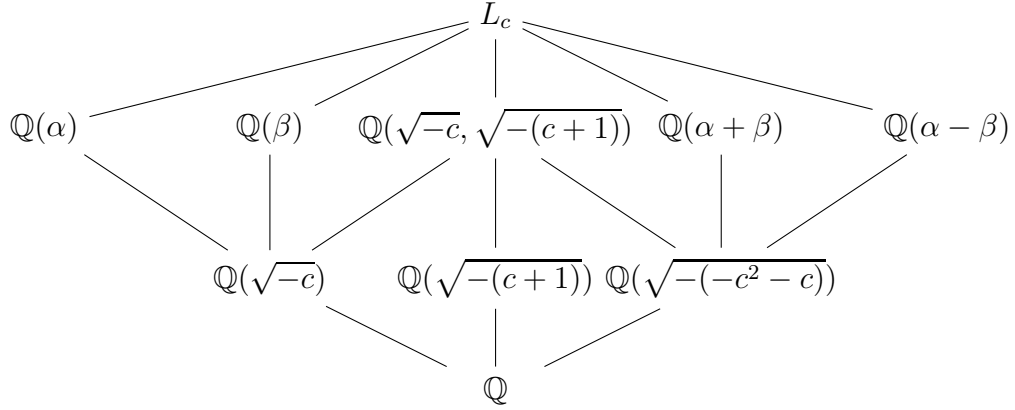
\begin{figure} \label{fig:lattice}
\begin{tikzpicture}[sibling distance=3cm]
\centering
\node (topnode) at (0,5) { $L_c$ } 
child { node {$\mathbb{Q}(\alpha)$}}
child { node {$\mathbb{Q}(\beta)$}} 
child { node {{$\mathbb{Q}(\sqrt{-c},\sqrt{-(c+1)}$)}} }
child { node {{$\mathbb{Q}(\alpha+\beta$)}} }
child { node {{$\mathbb{Q}(\alpha-\beta$)}} }
;
\node[minimum width=3cm](bottomnode) {$\mathbb{Q}$} [grow'=up]
child { node {$\mathbb{Q}(\sqrt{-c}$)} }
child { node {$\mathbb{Q}(\sqrt{-(c+1)}$)}}
child { node {$\mathbb{Q}(\sqrt{-(-c^2-c)})$} }
;
\foreach \x in {1,2,3}{
\draw (bottomnode-1) -- (topnode-\x);
}
\foreach \x in {3}{
\draw(bottomnode-2) -- (topnode-\x);}
\foreach \x in {3,4,5}{
\draw(bottomnode-3) -- (topnode-\x);}
\end{tikzpicture}
\caption{Subfield lattice for $L_c$ over $\Q$ when both $-c$ and $-(c+1)$ are non-squares.}
\end{figure}


Now when $-c$ is not a square in $\Z$ but $-(c+1)$ is a square, we have that $f_c^2(x)$ is irreducible and $L_c = \Q(\alpha)(\sqrt{-(c+1)}) = \Q(\alpha)$, and thus $[L_c : \Q] = 4$. But we have already shown that $\Gal(L_c/\Q)$ is isomorphic to $D_4$ or $\Z/4\Z$, and hence we must have $\Gal(L_c/\Q) \cong \Z/4\Z$.

Finally, when $-c$ is a square in $\Z$ and $c \neq 0, -1$, we may write $c = -b^2$ and we have the factorization in \eqref{factor}. Thus $L_c = \Q(\sqrt{b^2 + b}, \sqrt{b^2 - b})$, and we claim $[L_c : \Q] = 4$. It suffices to show that $b^2 + b$ is not a square in $\Q(\sqrt{b^2 - b})$, and a straightforward argument shows this holds provided $c \not\in \{-1, 0\}$. Thus in this case we see that $\Gal(L_c/\Q) \cong (\Z/2\Z \times \Z/2\Z)$.

Putting this all together, we have shown:

\begin{proposition} \label{galois}
Let $f_c(x) = x^2 + c$ for $c \in \Z \setminus \{-1, 0\}$, and let $L_c = \Q(\alpha, \beta)$ be the splitting field of $f_c^2(x)$ over $\Q$, where $\alpha$ and $\beta$ are as in \eqref{ab}. 
\begin{enumerate}
\item If neither $-c$ nor $-(c+1)$ is a square in $\Z$, then $L_c$ is a $D_4$-extension of $\Q$ with subfield lattice given in Figure 1.
\item If $-c$ is not a square in $\Z$ but $-(c+1)$ is a square in $\Z$, then $L_c = \Q(\alpha)$ and $L_c$ is a $\Z/4\Z$-extension of $\Q$ with unique quadratic subfield $\Q(\sqrt{-c})$.
\item If $c = -b^2 \in \Z \setminus \{-1, 0\}$, then $L_c = \Q(\sqrt{b^2 + b}, \sqrt{b^2 - b})$ and $L_c$ is a 
$(\Z/2\Z \times \Z/2\Z)$-extension of $\Q$. 
\end{enumerate}
\end{proposition}

\section{Factorization of the ideal $2\Oo_{L_c}$: Global methods} \label{global}


In this section we calculate decomposition and inertia groups in $\Gal(L_c/\Q)$ for many values of $c$, and use this to give the factorization of $2\Oo_{L_c}$ for many values of $c$. 

Let $L/K$ be any Galois extension of number fields, with $\Gal(L/K) = G$. Recall that the \textit{decomposition group} of $L/K$ at a prime $\B$ in $\Oo_{L}$ is defined to be
$$
D_\B = \{\sigma \in G : \sigma(\B) = \B\},
$$
where $G = \Gal(L/K)$. 
The \textit{inertia subgroup} at $\B$ is the kernel of the natural homomorphism sending $\sigma \in D$ to the map
$\overline{\sigma}$ on the residue field $\Oo_L / \B$. In other words, if $\B \cap K = \p$, then we have
$$
I_\B = \{\sigma \in G : \text{$\sigma(x) \equiv x \bmod{\B}$ for all $x \in \Oo_L$} \}.
$$
Clearly $I_\B$ is a subgroup of $G$, and it is straigtforward to check that in fact $I_\B$ is normal in $G$, with 
$D_\B/I_\B \cong \Gal((\Oo_L/\B)/(\Oo_K/\p))$. The latter is an extension of finite fields, and hence Galois with cyclic Galois group. 

The utility of $D_\B$ and $I_\B$ arises mainly through their connection to ideal factorizations such as the one in \eqref{genfact}. Indeed, $|D_\B| = ef$ (and hence $g = [G:D_\B]$), and $|I_\B| = e$. Moreover, the decomposition and inertia groups at primes above a given prime $\p$ are all conjugate subgroups of $G$. Finally, the fixed fields of $D_\B$ and $I_\B$ enjoy special properties: 
\begin{enumerate}[label=\Alph*]
\item[A.] \label{decomp} $L^{D_\B}$ contains all sub-extensions $E$ of $L/K$ in which $E \cap \B$ appears in the factorization of $\p\Oo_E$ without ramification or resiude extension; and 
\item[B.] \label{inertia} $L^{I_\B}$ contains all sub-extensions $E$ in which $E \cap \B$ appears in the factorization of $\p\Oo_E$ without ramification.
\end{enumerate}

Recall that if $t$ is an integer, $2\Oo_{\Q(\sqrt{-t})}$ splits if it factors as $\p_1 \p_2$ in $\Oo_{\Q(\sqrt{-t})}$, with $\p_1 \neq \p_2$, is inert if it remains prime in $\Oo_{\Q(\sqrt{-t})}$, and ramifies if it factors as $\p^2$ in $\Oo_{\Q(\sqrt{-t})}$. We observe:

\begin{proposition} \label{quadraticcase}
Let $t \in \Z$ with $-t$ not a square, and write $t = 4^ns$ with $4 \nmid s$. Then
\begin{equation*} 
2\Oo_{\Q(\sqrt{-t})} \;\;
\begin{dcases}
    								\text{splits} & \text{if $s \equiv 7 \bmod{8},$}\\
    								\text{is inert} &\text{if  $s \equiv 3 \bmod{8},$}\\
                                    \text{ramifies} & \text{otherwise.}
\end{dcases}
\end{equation*}
\end{proposition}

\begin{proof}
The proposition is well-known in the case where $s$ is squarefree, that is, not divisible by the square of a prime. Write $t = r^2S$, where $r \in \Z$ and $S \in \Z$ is squarefree. We prove the proposition by showing $s \equiv S \bmod{8}$. We have $4^ns = r^2S$, and so $s = (r/2^n)^2S$, where $(r/2^n) \in \Z$. Moreover, $(r/2^n)$ is odd, for otherwise $4 \mid s$. Thus $(r/2^n)^2 \equiv 1 \bmod{8}$, and hence $s \equiv S \bmod{8}$, as desired. 
\end{proof}

This is already enough to handle the biquadratic case.

\begin{proposition} \label{biquad}
Let $f_c(x) = x^2 + c$ for $c =  -b^2 \in \Z \setminus \{-1, 0\}$, and consider the factorization given in \eqref{genfact}. We have
\begin{align*}
e = 4, f = 1, g = 1 & \qquad \text{iff}  \; \begin{cases} 
\text{$b \equiv \pm 3, \pm 7, \pm 13 \bmod{32}$} \\ 
\text{$b \equiv \pm 15 \bmod{64}$} \\
\text{$b = \pm 4^{j}(4r + 2)$, where $j \geq 0$ and $r \in \Z$} \\
\text{$b = \pm (1 + 4^{k}(4r + 3))$, where $k \geq 2$ and $r \in \Z$} \\
\text{$b = \pm (1 + 4^{k}(8r + 6))$, where $k \geq 2$ and $r \in \Z$}
\end{cases}\\
e = 2, f = 2, g = 1 & \qquad \text{iff}  \; 
\begin{cases} 
\text{$b \equiv \pm 4, \pm 5 \bmod{32}$} \\
\text{$b \equiv \pm 9 \bmod{64}$} \\
\text{$b = \pm 4^{k}(8r + 3)$, where $k \geq 2$ and $r \in \Z$} \\ 
\text{$b = \pm (1 + 4^{k}(8r + 5))$, where $k \geq 2$ and $r \in \Z$} \\
\text{$b = \pm(1 + 4^{k}(16r + 10))$, where $k \geq 2$ and $r \in \Z$} 
\end{cases} \\
e = 2, f = 1, g = 2 & \qquad \text{iff	}  \; 
\begin{cases}
\text{$b \equiv \pm 11, \pm 12 \bmod{32}$} \\
\text{$b \equiv \pm 23 \bmod{64}$} \\
\text{$b = \pm 4^{k}(8r + 1)$, where $k \geq 2$ and $r \in \Z$} \\
\text{$b = \pm (1 + 4^{k}(8r + 1))$, where $k \geq 2$ and $r \in \Z$} \\
\text{$b = \pm(1 + 4^{k}(16r + 2))$, where $k \geq 2$ and $r \in \Z$} 
\end{cases} \\
\end{align*}
and these are the only possibilities.
\end{proposition}

\begin{proof}
From part (3) of Proposition \ref{galois}, it follows that $L_c$ has precisely three quadratic subextensions, namely $\Q(\sqrt{b^2 + b}), \Q(\sqrt{b^2 - b})$, and $\Q(\sqrt{b^2 - 1})$. First we show that $e \geq 2$ in \eqref{genfact}. If $b \equiv 1, 2 \bmod{4}$, then $b^2 + b \equiv 2 \bmod{4}$, and it follows from Proposition \ref{quadraticcase} that $(2)$ ramifies in $\Q(\sqrt{b^2 + b})$, and hence in $L_c$. If $b \equiv 3 \bmod{4}$, then $b^2 - b \equiv 2 \bmod{4}$, and if $b \equiv 0 \bmod{4}$, then $-(b^2 - 1) \equiv 1 \bmod{4}$, and thus $(2)$ ramifies in $\Q(\sqrt{b^2 - b})$ and $\Q(\sqrt{b^2 - 1})$, respectively. Hence in all cases $(2)$ ramifies in $L_c$. 

We've thus shown that if $\B$ is a prime of $L_c$ lying over $(2)$, then either $L_c^{I_B} = \Q$ (equivalently, $(2)$ is totally ramified in $L_c$) or $L_c^{I_B}$ is a quadratic sub-extension of $\Q$. We examine the cases where the latter holds, and then the complement of those cases gives the totally ramified case. 

First note that $L_c^{I_B} = \Q(\sqrt{b^2 + b})$ if and only if $(2)$ either splits or is inert in $\Q(\sqrt{b^2 + b})$. By Proposition \ref{quadraticcase}, $(2)$ splits in $\Q(\sqrt{b^2 + b})$ precisely when 
\begin{equation} \label{toshow}
b^2 + b = 4^ns \qquad \text{where $n \geq 0, 4 \nmid s$, and $s \equiv 1 \bmod{8}$.}
\end{equation}
Write 
$$b = 4^kv \qquad \text{with $k \geq 0$ and $4 \nmid v$.} 
$$
and note that 
\begin{equation} \label{firsteqn}
b^2 + b = b(b+1) = 4^k(4^kv^2 + v).
\end{equation}
We now divide our considerations into two cases.

{\noindent \textbf{Case 1: $4$ divides $4^kv^2 + v$}}. Because $4 \nmid v$, this can only happen if $k = 0$ and $4 \mid (v+1)$. Hence we write $v = -1 + 4^\ell m$ with $\ell \geq 1$ and $4 \nmid m$. Then
$$b^2 + b = v(v+1) = 4^{\ell}[m(-1 + 4^\ell m)].$$
By assumption $4 \nmid m(-1 + 4^\ell m)$ and hence \eqref{toshow} holds if and only if $m(-1 + 4^\ell m) \equiv 1 \bmod{8}$. When $\ell = 1$, this is equivalent to $m(-1 + 4m) \equiv 1 \bmod{8}$, which holds if and only if $m \equiv 3 \bmod{8}$, i.e. $b \equiv 11 \bmod{32}$. When $\ell \geq 2$, we have $m(-1 + 4^\ell m) \equiv 1 \bmod{8}$ if and only if $-m \equiv 1 \bmod{8}$, i.e. $b = -1 + 4^{\ell}(8r + 7)$, with $\ell \geq 2$. 

\smallskip

{\noindent \textbf{Case 2: $4$ does not divide $4^kv^2 + v$}}.
In this case\eqref{toshow} holds if and only if $4^kv^2 + v \equiv 1 \bmod{8}$, i.e. 
\begin{equation} \label{new}
v \equiv 1 - (2^kv)^2 \bmod{8}.
\end{equation}
Now \eqref{new} holds if and only if we have one of the following:
\begin{enumerate}
\item $2^kv$ odd and $v \equiv 0 \bmod{8}$
\item $2^kv \equiv 2 \bmod{4}$ and $v \equiv 5 \bmod{8}$
\item $2^kv \equiv 0 \bmod{4}$ and $v \equiv 1 \bmod{8}$
\end{enumerate}
The first is impossible. The second holds precisely when $k = 1$ and $v \equiv 5 \bmod{8}$, i.e. $b \equiv 20 \bmod{32}$. The third holds precisely when $k \geq 2$ and $v \equiv 1 \bmod{8}$, i.e. $b = 4^k(8r + 1)$ for $r \in \Z$.


The cases where $(2)$ is inert in $\Q(\sqrt{b^2 + b})$ and $L_c^{I_B} = \Q(\sqrt{b^2 - b})$ are handled similarly. We thus turn to the case where $L_c^{I_B} = \Q(\sqrt{b^2 - 1})$. If $b \equiv 0 \bmod{4}$, then $-(b^2 - 1) \equiv 1 \bmod{4}$, and hence $(2)$ ramifies in $\Q(\sqrt{b^2 - 1})$. Thus we may assume $b \not\equiv 0 \bmod{4}$. If $4 \nmid b^2 - 1$, then from Proposition \ref{quadraticcase} we have $b^2 - 1 \equiv 1 \bmod{4}$, which is impossible. Hence we have $4 \mid b^2 - 1$, and thus $b$ is odd. Write 
$$b = 1 + 2^km, \quad \text{where $k \geq 1$ and $m$ is odd},$$ 
and note that 
$$
b^2 - 1 = 2^{k + 1} m (1 + 2^{k - 1}m).
$$
We now consider two cases.

{\noindent \textbf{Case 1: $k \geq 2$}}. In this case 
$(1 + 2^{k - 1}m)$ is odd, and so from Proposition \ref{quadraticcase} we must have $k$ odd and 
\begin{equation} \label{next}
m (1 + 2^{k - 1}m) \equiv 1 \bmod{8} \qquad m (1 + 2^{k - 1}m) \equiv 5 \bmod{8},
\end{equation}
corresponding to the split and inert case, respectively. Writing $k = 2\ell + 1$ for $\ell \geq 1$ and noting that $m^2 \equiv 1 \bmod{8}$, we have that \eqref{next} is equivalent to 
$$m \equiv 1 - 2^{2\ell} \bmod{8} \qquad m \equiv 5 - 2^{2 \ell} \bmod{8}.$$
When $\ell = 1$ we obtain $m \equiv 5 \bmod{8}$ and $m \equiv 1 \bmod{8}$, respectively; these correspond to $b \equiv 41 \bmod{64}$ and $b \equiv 9 \bmod{64}$, respectively. 
When $\ell \geq 2$ we have $m \equiv 1 \bmod{8}$ and $m \equiv 5 \bmod{8}$, respectively, which correspond to $b = 1 + 4^{\ell}(16r + 2)$ and $b = 1 + 4^{\ell}(16r + 10)$, respectively. 

{\noindent \textbf{Case 2: $k = 1$}}. In this case, $b = 1 + 2m$, so $b^2 - 1 = 4(m + m^2)$. If $m \equiv 1 \bmod{4}$, then $m^2 + m \equiv 2 \bmod{4}$, and hence $(2)$ ramifies in $\Q(\sqrt{b^2 - 1})$, a contradiction.
Thus we write $m = 3 + 4^{\ell}t$, with $\ell \geq 1$ and $4 \nmid t$, whence
$$
b^2 - 1 = 4(m + m^2) = 4^2(3 + 4^\ell t)(1 + 4^{\ell - 1}t).
$$
If $\ell \geq 2$, then $(3 + 4^\ell t)(1 + 4^{\ell - 1}t) \equiv 3 \bmod{4}$, which again gives a contradiction by Proposition \ref{quadraticcase}. Thus $\ell = 1$ and 
\begin{equation} \label{latest}
b^2 - 1 = 4^2(3 + 4t)(1 + t).
\end{equation}
If $t \equiv 1 \bmod{4}$, then $(3 + 4t)(1 + t) \equiv 2 \bmod{4}$, again giving a contradiction. If $t \equiv 2 \bmod{4}$, then write $t = 2s$ with $s$ odd, and note
$$(3 + 4t)(1 + t) = 3 + 7t + 4t^2 \equiv 3 + 6s \bmod{8}.$$
If $s \equiv 1 \bmod{4}$, then this yields $1$ modulo $8$, and hence (2) splits in $\Q(\sqrt{b^2 - 1}$. This gives $b \equiv 23 \bmod{64}$. If $s \equiv 3 \bmod{4}$, then (2) is inert in $\Q(\sqrt{b^2 - 1})$. This gives $b \equiv 55 \bmod{64}$. 

If $t \equiv 3 \bmod{4}$, then write $t = -1 + 4^qs$ with $q \geq 1$ and $4 \nmid s$. Then from \eqref{latest} we obtain
$$
b^2 - 1 = 4^{q+2}[(-1 + 4^{q+1}s)s],
$$
and since $4 \nmid s$ we have that $4 \nmid (-1 + 4^{q+1}s)s$. Now since $q \geq 1$ we have
$(-1 + 4^{q+1}s)s \equiv -s \bmod{8}$, and so from Proposition \ref{quadraticcase} we have that $(2)$ splits in $\Q(\sqrt{b^2 - 1})$ if $s \equiv 7 \bmod{8}$ and that $(2)$ is inert if $s \equiv 3 \bmod{8}$. These give $b = -1 + 4^{q+1}(16r + 14)$ and 
$b = -1 + 4^{q+1}(16r + 6)$, respectively.
\end{proof}

\section{Proof of Theorem \ref{main1} in the cases $e = 8, f = 1, g = 1$ and $e = 4, f = 2, g = 1$} \label{mainglob}



We now use our knowledge of subfields of $L_c$ to limit the possibilities for $L_c^{D_\B}$ and $L_c^{I_\B}$, where $\B$ is an ideal of $L_c$ lying over $(2)$. This obviously depends on the structure of $\Gal(L_c/\Q)$ which by Proposition \ref{galois} depends on whether $-c$ and $-(c+1)$ are squares in $\Z$. We discuss here the case where neither $-c$ nor $-(c+1)$ is a square and there is a unique prime ideal of $\Oo_{L_c}$ lying above $(2)$, thereby proving the corresponding parts of Theorem \ref{main1}. We leave the cases where one of $-c$ or $-(c+1)$ is a square to Section \ref{local}.

Because neither $-c$ nor $-(c+1)$ is a square, the subfield lattice of $L_c$ is given in Figure 1. 
Note that all subextensions are Galois save $\Q(\alpha)$ and $\Q(\beta)$, which together form a Galois-conjugacy class, and $\Q(\alpha+\beta)$, $\Q(\alpha - \beta)$, which form another Galois conjugacy class. 
One notices immediately from Figure 1 that each sub-extension of $L_c/\Q$ of degree at least two over $\Q$ contains one of the three quadratic sub-extensions 
\begin{equation} \label{quadratics}
\Q\left(\sqrt{-c}\right), \Q\left(\sqrt{-(c+1)}\right), \; \text{and} \; \Q\left(\sqrt{-(-c^2-c)}\right).
\end{equation}
The odd-looking expressions for the last two of these fields are helpful in light of Proposition \ref{quadraticcase}.



In order to prove a classification theorem such as Theorem \ref{main1}, one is confronted with the problem of selecting an indexing quantity. The most obvious choice in the present setting is the parameter $c$, but as one can see from the complexity of the statement of Theorem \ref{main1}, this leads to a dizzying number of cases. A more profitable choice is $L_c^{I_\B}$, though here too one encounters significant complexities, for instance in distinguishing between the cases where $[L_c^{I_\B} : \Q]$ has degree two or four. We thus adopt the following hybrid approach: in this section we use global methods to analyze the cases where $L_c^{D_\B} = \Q$, and in the following section we use local methods to analyze the remaining cases; in this latter section we index by the behavior of $L_c^{I_\B}$. The advantage to the supposition that $L_c^{D_\B} = \Q$ is that we may determine the factorization of $2\Oo_{L_c}$ by analyzing only the quadratic sub-extensions of $L_c$.

We remark that under the assumption $L_c^{D_\B} = \Q$ it is easy to see that $[L_c^{I_\B} : L_c^{D_\B}]$ must be 1 or 2. Indeed, If $\B$ is the unique prime of $\Oo_{L_c}$ above $2$, then $L_c^{I_\B}$ is a Galois extension of $L_c^{D_\B}$, and hence of $\Q$, necessarily cyclic. Hence $\Gal(L_c^{I_\B}/\Q)$ is a cyclic quotient of $D_4$, and thus has order $1$ or $2$. In Corollary \ref{residueprop} we use Newton polygons to prove the more general fact that $f \leq 2$ for all $c$ with $-c$ not a square. 

\medskip

\noindent \textbf{Case 1: $L_c^{D_\B} = L_c^{I_\B} = \Q$ for some prime $\B$ of $\Oo_{L_c}$ lying above 2.} 

This is equivalent to $e = 8, f = 1, g = 1$; in particular, $\B$ is the only prime of $\Oo_{L_c}$ lying over $2$. In light of property (B) on p. \pageref{inertia}, $L_c^{I_\B} = \Q$ holds if and only if $2$ ramifies in \textit{every} subextension of $L_c$. This in turn is equivalent to $2$ ramifying in each of the three quadratic sub-extensions given in \eqref{quadratics}; one direction of this is obvious, and the other follows from the fact that an extension is ramified at 2 if any of its sub-extensions is. 

Hence from Proposition \ref{quadraticcase} we may write 
\begin{equation} \label{initialcase1}
\text{$c  = 4^js$} \qquad \text{$c+1  = 4^ms'$},
\end{equation}
where $s$ and $s'$ are both $1$ or $2$ modulo $4$. Because $c$ and $c+1$ are relatively prime, at least one of $s$ and $s'$ must be $1$ modulo $4$, and because $2$ ramifies in $\Q(\sqrt{c^2 + c})$, we cannot have 
$-ss' \equiv 3 \bmod{4}$, and hence at most one of $s$ and $s'$ can be $1$ modulo $4$. We thus have two cases to consider:
\begin{enumerate} 
\item[(a)] $s \equiv 1 \bmod{4}, \quad s' \equiv 2 \bmod{4}$    
\item[(b)] $s \equiv 2 \bmod{4}, \quad s' \equiv 1 \bmod{4}$ 
\end{enumerate}
Suppose that we are in case (a), and assume first that $j = 0$ in \eqref{initialcase1}. Then we may write $c = 4r + 1$ for some $r \in \Z$, and we obtain $c+1 = 4r + 2$. It follows that $m = 0$ in \eqref{initialcase1}, and we get $s' = c+1 \equiv 2 \bmod{4}$.
If $j > 0$, then we have $m = 0$ and $s' = c + 1 \equiv 1 \bmod{4}$, contradicting our assumption of being in case (a). We have thus shown that case (a) holds if and only if $c \equiv 1 \bmod{4}$. 

Suppose now that we are in case (b), and assume first that $j = 0$ in \eqref{initialcase1}. Then $c + 1 \equiv 3 \bmod{4}$, which forces $m = 0$ in \eqref{initialcase1} and gives the contradiction $s' \equiv 3 \bmod{4}$. If $j > 0$, then must have $m = 0$. We may write $c = 4^j(4r + 2)$, and this gives $s' = c + 1 \equiv 1 \bmod{4}$. Hence case (b) holds if and only if 
$$c = 4^j(4r + 2) = 2^{2j + 1}m,$$
where $m = 2r + 1$ is an arbitrary odd number. 

\medskip

\noindent \textbf{Case 2: $L_c^{D_\B} = \Q$ and $[L_c^{I_\B} : L_c^{D_\B}] = 2$ for some prime $\B$ of $\Oo_{L_c}$ lying above 2.} 

This is equivalent to $e = 4, f = 2, g = 1$; in particular, $\B$ is the only prime of $\Oo_{L_c}$ lying over $2$. In light of property (B) on p. \pageref{inertia}, this occurs if and only if $2$ is inert in one of the quadratic sub-extensions of $L_c$, and $2$ ramifies in the other two such sub-extensions. 

\textbf{Sub-Case 2a: $L_c^{I_\B} = \Q(\sqrt{-c})$.}
From Proposition \ref{quadraticcase} this holds if and only if 
\begin{equation} \label{initialcase2}
\text{$c  = 4^js$ and $c+1  = 4^ms'$, where $s \equiv 3 \bmod{8}$ and $s' \equiv 1, 2 \bmod{4}$.}
\end{equation}
Note that under these assumptions, we have $-ss' \equiv s' \bmod{4}$, and hence $2$ ramifies in $\Q(\sqrt{c^2 + c})$. 

Assume first that $j = 0$ in \eqref{initialcase2}. Write $c = 8r + 3$, whence $c + 1 = 4(2r + 1)$. Thus $m = 1$ and $s' = 2r + 1 \equiv 1 \bmod{4}$ in \eqref{initialcase2}. This holds if and only if $r$ is even, i.e., $c \equiv 3 \bmod{16}$. If $j > 0$ in \eqref{initialcase2}, then write $c = 4^j(8r+3)$,and note that we automatically have $c + 1 \equiv 1 \bmod{4}$, and thus $s' = c+1 \equiv 1 \bmod{4}$. Hence Sub-Case 2a holds if and only if $c \equiv 3 \bmod{16}$ or $c = 2^{2j}(8r+3)$.

\textbf{Sub-Case 2b: $L_c^{I_\B} = \Q(\sqrt{-(c+1)})$.}
From Proposition \ref{quadraticcase} this holds if and only if 
\begin{equation} \label{initialcase3}
\text{$c  = 4^js$ and $c+1  = 4^ms'$, where $s \equiv 1, 2 \bmod{4}$ and $s' \equiv 3 \bmod{8}$.}
\end{equation}
Note that under these assumptions, we have $-ss' \equiv s \bmod{4}$, and hence $2$ ramifies in $\Q(\sqrt{c^2 + c})$. 

Assume first that $m = 0$ in \eqref{initialcase3}. Write $c + 1 = 8r + 3$, which is equivalent to $c \equiv 2 \bmod{8}$. In this situation $j = 0$ and $c = s = \equiv 2 \bmod{4}$, so all conditions in \eqref{initialcase3} are satisfied. 
If $m > 0$ in \eqref{initialcase3}, then we have $c \equiv 3 \bmod{4}$, and hence $j = 0$ and $s \equiv 3 \bmod{4}$, a contradiction. Therefore Sub-Case 2b holds if and only if $c \equiv 3 \bmod{16}$ or $c = 2^{2j}(8r+3)$ for some $j \geq 1$. To make this match better with the conclusion of Sub-Case 2c, we split the latter into two families: $c \equiv 12 \bmod{32}$ (i.e. $j = 1$) and $c = 2^{2j}(8r+3)$ for some $j \geq 2$.

\textbf{Sub-Case 2c: $L_c^{I_\B} = \Q(\sqrt{c^2 + c})$.}
From Proposition \ref{quadraticcase} this holds if and only if 
\begin{equation} \label{initialcase4}
\text{$c  = 4^js$ and $c+1  = 4^ms'$, where $-ss' \equiv 3 \bmod{8}$ and $s, s' \equiv 1, 2 \bmod{4}$.}
\end{equation}
Note that under these assumptions, we must have $s \equiv s' \equiv 1 \bmod{4}$. Moreover, we have $-ss' \equiv 3 \bmod{8}$ if and only if either 
\begin{enumerate}
\item[(a)] $s \equiv 1 \bmod{8}$ and $s' \equiv 5 \bmod{8}$, or 
\item[(b)] $s \equiv 5 \bmod{8}$ and $s' \equiv 1 \bmod{8}$.
\end{enumerate}

Suppose first that we are in case (a), and assume that $j = 0$ in \eqref{initialcase4}. Then $s = c \equiv 1 \bmod{4}$, so $c + 1 \equiv 2 \bmod{4}$, contradicting our supposition that $s' \equiv 5 \bmod{8}$. If $j > 0$, then we have $c = 4^j(8r + 1)$, and thus $c + 1 = 4^j(8r + 1) + 1$, giving $m = 0$ and 
$$
4^j(8r + 1) + 1 = s' \equiv 5 \bmod{8},
$$
which holds if and only if $j = 1$, i.e. $c \equiv 4 \bmod{32}$. 

Now suppose that we are in case (b), and assume that $j = 0$ in \eqref{initialcase4}. As in the preceding paragraph, we have $s' = c+1 \equiv 2 \bmod{4}$, a contradiction. If $j > 0$, then we have $c = 4^j(8r + 5)$, and thus $c + 1 = 4^j(8r + 5) + 1$, giving $m = 0$ and 
$$
4^j(8r + 5) + 1 = s' \equiv 1 \bmod{8},
$$
which holds if and only if $j \geq 2$. Therefore Sub-Case 2c holds if and only if $c \equiv 4 \bmod{32}$ or $c = 2^{2j}(8r-3)$ for some $j \geq 2$.

\section{Factorization of the ideal $2\Oo_{L_c}$: Local methods} \label{local}

We now make use of the fact that if $K$ is any number field, then we have
\begin{equation} \label{2adicdecomp}
K \otimes_\Q \Q_2 \cong \prod_{i = 1}^g K_{2,i},
\end{equation}
where the $K_{2,i}$ are finite extensions of the $2$-adic numbers $\Q_2$ given as follows: write $K = \Q(\gamma)$ and let $f$ be the minimal polynomial for $\gamma$ over $\Q$. Using the natural embedding $\Q \hookrightarrow \Q_2$, consider $f$ as having coefficients in $\Q_2$, and suppose that $f_1f_2 \cdots f_g$ is the factorization of $f$ into irreducibles in the ring $\Q_2[x]$. Then $K_{2,i} = \Q_2(\gamma_i)$, where $\gamma_i$ is a root of $f_i$ in the algebraic closure of $\Q_2$. (See \cite[pp. 25-26]{FT} for details.) Moreover the fields $K_{2,i}$ encode crucial information about ideal factorizations: if
$2\Oo_K = \p_1^{e_1} \cdots \p_g^{e_g}$, then the value group of $K_{2,i}$ is $(1/e_i)\Z$ and the residue extension degree of $K_{2,i}$ is the same as $|(\Oo_K/\p_i) : (\Z/2\Z)|$.

Throughout this section, $v : \Q_2 \to \Z$ denotes the $2$-adic valuation, and $| \cdot |$ denotes the 2-adic absolute value.
A very useful tool for understanding the fields $K_{2,i}$ given in \eqref{2adicdecomp} is the $2$-adic Newton polygon of a polynomial $f(x) = \sum_{i = 0}^n a_ix^i$, namely the polygon given by taking the lower convex hull of the points $(i, v(a_i))$. We assume the reader is familiar with the relationship between slopes of the $2$-adic Newton polygon of a polynomial and the $2$-adic valuation of the polynomial's roots (see e.g. \cite[Theorem 5.11]{jhsdynam}).

We begin by using Newton polygons to prove Theorem \ref{main0}:

\begin{theorem} \label{ram}
For all $c \in \Z \setminus \{0\}$, we have $e \geq 2$ in \eqref{genfact}.
\end{theorem}

\begin{proof}
When $c = -1$, we have $L_c = \Q(\sqrt{2})$, which has $e = 2$. When $-c$ is a square in $\Z$ and $c \not\in \{-1,0\}$, the result follows from the first paragraph of the proof of Proposition \ref{biquad}. 

Assume now that $-c$ is not a square in $\Z$. When $c \equiv 3 \bmod{4}$, 
the Newton polygon of $f_c^2 = x^4 + 2cx^2 + c^2 + c$ contains the segment connecting $(4,0)$ and $(2,1)$ which has slope $-1/2$. Thus the value group of $L_c$ contains $(1/2)\Z$, and it follows that $e \geq 2$. When $c \equiv 1 \bmod{4}$, it follows similarly that $e \geq 4$. When $c \equiv 2 \bmod{4}$, we have from Proposition \ref{quadraticcase} that $(2)$ ramifies in $\Q(\sqrt{-c})$, and hence in $L_c$. Suppose now that $c \equiv 0 \bmod{4}$. If $-(c+1)$ is a square in $\Z$, then we may write $c = -(b^2 + 1)$, which cannot be congruent to $0$ modulo $4$. Therefore $-(c+1)$ is not a square in $\Z$, and $\Q(\sqrt{-(c+1)})$ is a quadratic extension of $\Q$. But $c+1 \equiv 1 \bmod{4}$, and it follows from Proposition \ref{quadraticcase} that $(2)$ ramifies in $\Q(\sqrt{-(c+1)})$, and hence in $L_c$.
\end{proof}




\begin{corollary} \label{residueprop}
For all $c \in \Z$, we have $f \leq 2$ in \eqref{genfact}.
\end{corollary}

\begin{proof}
When $[L_c : \Q] \leq 4$, this follows immediately from Theorem \ref{ram}. Thus assume that $[L_c : \Q] = 8$, and hence $\Gal(L_c/\Q) \cong D_4$. Suppose that $f \geq 4$ in \eqref{genfact}. Then from Theorem \ref{ram} we must have $e = 2, f = 4,$ and $g = 1$. Because $g = 1$, there is a unique prime $\B$ of $\Oo_{L_c}$ above $2$. Because $L_c^{D_\B} = \Q$, we have that $L_c^{I_\B}$ is a Galois extension of $\Q$, necessarily cyclic, and of degree at least 4. Hence $\Gal(L_c^{I_\B}/\Q)$ is a cyclic quotient of $D_4$ of order at least four, which does not exist.
\end{proof}

We now have all the tools we need to describe the factorization of $2\Oo_{L_c}$ in the case where $-c$ is not a square in $\Z$ but $-(c+1)$ is.

\begin{proposition} \label{cyclic}
Let $f_c(x) = x^2 + c$ where $-c$ is not a square in $\Z$ but $-(c+1)$ is a square in $\Z$, and write $c = -(b^2 + 1)$ for $b \in \Z$. In the factorization given in \eqref{genfact}, we have 
\begin{align*}
e = 4, f = 1, g = 1 & \; \text{iff $b$ is odd,}  \\
e = 2, f = 2, g = 1 & \; \text{iff $b \equiv 2 \bmod{4}$,} \\
e = 2, f = 1, g = 2 & \; \text{iff $b \equiv 0 \bmod{4}$,}  
\end{align*}
and these are the only possibilities.
\end{proposition}

\begin{proof}
By part (2) of Proposition \ref{galois}, $L_c$ is a $\Z/4\Z$-extension of $\Q$ with unique quadratic sub-extension $\Q(\sqrt{-c})$. If $2$ ramifies in $\Q(\sqrt{-c})$, then from property (B) on p. \pageref{inertia} we have that $L_c^{I_\B} = \Q$ for some prime $\B$ of $\Oo_{L_c}$ lying above $2$, and hence $e = 4$. If $2$ splits (resp. is inert) in $\Q(\sqrt{-c})$, then by Theorem \ref{ram} we have $e = 2$, $f = 1$, $g = 2$ (resp. $e = 2$, $f = 2$, $g = 1$). The Proposition now follows from Proposition \ref{quadraticcase}.
\end{proof}

We need to study the squaring map on $\Q_2$, and especially its inverse. First we record the well-known fact that the squaring map sends $1 + 2\Z_2$ surjectively onto $1 + 8\Z_2$. More generally, if $S$ is the squaring map, then 
\begin{equation} \label{square image}
S(\Q_2) = \{0\} \cup \bigcup_{n \in \Z} 4^n(1 + 8\Z_2)
\end{equation}
See for instance \cite[p. 85]{FT}.

We also have a useful reformulation of Proposition \ref{quadraticcase}: let $t \in \Q_2$ with $-t$ not a square in $\Q_2$, and write $t = 4^ns$, with $s \in \Z_2$ and $v(s) \in \{0,1\}$. Then 
\begin{equation} \label{localquadraticcase}
\Q_2(\sqrt{-t}) = 
\begin{dcases}
    								\Q_2 & \text{if $s \equiv 7 \bmod{8},$}\\
    								\text{an unramified quadratic extension of $\Q_2$} & \text{if  $s \equiv 3 \bmod{8},$}\\
                                    \text{a ramified quadratic extension of $\Q_2$} & \text{otherwise.}
\end{dcases}
\end{equation}

The following proposition follows from Newton's binomial theorem, and will be useful in our later calculations.

\begin{proposition} \label{sqrt expansion}
If $v(x) > 2$, then in $\Q_2$ we have
$$\sqrt{1 + x} = 1 + \frac{1}{2}x - \frac{1}{8}x^2 + s2^r,$$
where $s \in \Z_2$ and $r = 3v(x) - 4$. 
\end{proposition}





\section{Proof of the remaining cases of Theorem \ref{main1}} 
\label{mainloc}

We now use the work in Section \ref{local} to examine the cases of Theorem \ref{main1} where $(e,f,g) = (2, 1, 4)$ and $(e,f,g) = (2,2,2)$. Let $\B$ be a prime of $\Oo_{L_c}$ lying above 2. In both of the cases we wish to study, $L_c^{I_\B}$ has degree 4 over $\Q$. What separates them is that in the first case, $L_c^{I_\B} = L_c^{D_\B}$, while in the second, $L_c^{D_\B}$ is a sub-extension of $L_c^{I_\B}$ of degree 2 over $\Q$. Up to Galois conjugation, there are three possibilities for $L_c^{I_\B}$, namely $\Q(\alpha)$, $\Q(\alpha + \beta)$, and $\Q(\sqrt{-c}, \sqrt{-(c+1)})$. 

\textbf{Case 1: $L_c^{I_\B} = \Q(\alpha)$.} \label{case1}
This means that one of the primes of $\Oo_{\Q(\alpha)}$ lying above 2 does not ramify. From \eqref{2adicdecomp} and the discussion following, this occurs if and only if one of the direct summands of $\Q(\alpha) \otimes_\Q \Q_2$ has value group $\Z$. But from Corollary \ref{residueprop} we know that $f \leq 2$, so either one summand is $\Q_2$ and we are in the $(2, 1, 4)$ case, or one summand is an unramified quadratic extension of $\Q_2$, and we are in the $(2,2,2)$ case.
Thus over $\Q_2$, either $f$ has a linear factor that generates a summand isomorphic to $\Q_2$, or $f$ has an irreducible quadratic factor that generates an unramified extension of $\Q_2$. We thus wish to study the roots of $f$ over $\Q_2$, and determine when one of them lies in $\Q_2$ or one of them generates an unramified quadratic extension of $\Q_2$. 

Note further that we must have $\Q(\sqrt{-c}) \subseteq L_c^{D_\B}$, since the former is the unique degree-2 sub-extension of $\Q(\alpha)$, and the latter has degree at least two. Hence $2$ must split in $\Q(\sqrt{-c})$, which is equivalent to $\Q_2(\sqrt{-c}) = \Q_2$.

Now from \eqref{square image}, we have that
$\sqrt{-c} \in \Q_2$ if and only if 
$$-c = 4^n(1 + 8b) \; \text{with $b \in \Z$ and $n \geq 0$.}$$ 
Assume this is indeed the case. From Proposition \ref{sqrt expansion}, we have that in $\Z_2$,
\begin{align}
\nonumber -c \pm \sqrt{-c} = 4^n(1 + 8b) \pm \sqrt{4^n(1 + 8b)} & = 
2^n\left(2^n(1 + 8b) \pm \sqrt{1 + 8b}\right) \\
\label{sqroots} 
& \equiv  2^n\left(2^n(1 + 8b) \pm (1 + 4b - 8b^2) \right) \bmod{2^{n+5}}.
\end{align}
Denote the numbers in \eqref{sqroots} by $\alpha_1$ and $\beta_1$, respectively. Note that the roots of $f(x)$ over $\Q_2$ can be taken to be $\pm \sqrt{\alpha_1}$ and $\pm \sqrt{\beta_1}$. We wish to study the extensions $\Q_2(\sqrt{\alpha_1})$ and $\Q_2(\sqrt{\beta_1})$, using \eqref{localquadraticcase}.

If $n$ is odd, both $\alpha_1$ and $\beta_1$ have odd 2-adic valuation, and hence from \eqref{localquadraticcase} both $\Q_2(\sqrt{\alpha_1})$ and $\Q_2(\sqrt{\beta_1})$ are ramified, meaning that every root of $f$ over $\Q_2$ generates a ramified extension of $\Q_2$, contrary to our original supposition. Thus we write $n = 2t$, and without loss of generality we have
\begin{align*}
\alpha_1 \equiv 2^{2t}(1 + 8b) + (1 + 4b - 8b^2) \equiv (2^{2t} + 1) + 4b(2^{2t+1} + 1) - 8b^2 \bmod{32}, \\
\beta_1 \equiv 2^{2t}(1 + 8b) - (1 + 4b - 8b^2)
\equiv (2^{2t} - 1) + 4b(2^{2t+1} - 1) + 8b^2 \bmod{32}. 
\end{align*}

We consider various values of $t$ separately. If $t = 0$, then 
\begin{align*}
\alpha_1 & \equiv 2 \bmod{4}, \\
\beta_1 & \equiv 4b + 8b^2 \bmod{32}. 
\end{align*}
Thus $\Q_2(\sqrt{\alpha_1})$ is ramified by \eqref{localquadraticcase}. 

On the other hand $\Q_2(\sqrt{\beta_1}) = \Q_2$ if and only if $\Q_2(\sqrt{-(-\beta_1})) = \Q_2$, which from \eqref{localquadraticcase} occurs if and only if $-\beta_1 = 4^n(-1 + 8r')$ for some $r' \in \Z_2$, or equivalently $b + 2b^2 = 4^n(1 + 8r)$ for some $r \in \Z_2$. Recall that in fact $b \in \Z$, and write $b = 2^jm$, where $m$ is odd. Then 
\begin{equation} \label{bthing}
b + 2b^2 = 2^j(m + 2^{j+1}m^2),
\end{equation}
so $j$ must be even and $m-1 + 2^{j+1}m^2$ must be a multiple of 8. Hence either $j = 0$ and $m \equiv 7 \bmod{8}$ or $j \geq 2$ and $m \equiv 1 \bmod{8}$. These are equivalent to, respectively, $b = 8r + 7$ and $b = 4^{k}(8r+1)$ for $k \geq 1$. Finally, these correspond to 
$c = -1 - 8(8r+7)$, or $c = 7 + 64r$ for some $r \in \Z$; and $c = -1 - 4^k(64r+8)$, or $c = -1 + 4^k(64r - 8)$ for some $k \geq 1, r \in \Z$.  

Similarly, we have from \eqref{localquadraticcase} that $\Q_2(\sqrt{\beta_1})$ is an unramified quadratic extension of $\Q_2$ if and only if $b + 2b^2 = 4^n(5 + 8r)$ for some $r \in \Z_2$. Proceeding as in the previous paragraph, we now have from \eqref{bthing} that $j$ must be even and $m-5 + 2^{j+1}m^2$ must be a multiple of 8. Hence either $j = 0$ and $m \equiv 3 \bmod{8}$ or $j \geq 2$ and $m \equiv 5 \bmod{8}$. These are equivalent to, respectively, $b = 8r + 3$ and $b = 4^{k}(8r+5)$ for $k \geq 1$. Finally, these correspond to 
$c = -1 - 8(8r+3)$, or $c = 39 + 64r$ for some $r \in \Z$; and $c = -1 - 4^k(64r+40)$, or $c = -1 + 4^k(64r + 24)$ for some $k \geq 1, r \in \Z$.   

If $t \geq 1$, then $v(\alpha_1) = v(\beta_1) = 0$. Thus from \eqref{localquadraticcase} we need only consider $\alpha_1$ and $\beta_1$ modulo 8. We have
\begin{align*}
\alpha_1 & \equiv 2^{2t} + 1 + 4b \bmod{8}, \\
\beta_1 & \equiv 2^{2t} - 1 - 4b \bmod{8}. 
\end{align*}

When $t = 1$, we have $\beta_1 \equiv 3 + 4b \bmod{8}$, and writing $\Q_2(\sqrt{\beta_1})$ as $\Q_2(\sqrt{-(-\beta_1)})$, we have from \eqref{localquadraticcase} that $\Q_2(\sqrt{\beta_1})$ is a ramified extension of $\Q_2$. On the other hand, we have $\alpha_1 \equiv 5 + 4b \bmod{8}$, and thus from \eqref{localquadraticcase} we see that $\Q_2(\sqrt{\alpha_1}) = \Q_2$ if $4b \equiv 4 \bmod{8}$ (i.e. $b$ is odd) and $\Q_2(\sqrt{\alpha_1})$ is an unramified quadratic extension of $\Q_2$ if $b$ is even. These correspond, respectively, to the cases where $-c = 16(1 + 8(2r+1))$ and $-c = 16(1 + 16r)$, which are equivalent to $c \equiv 112 \bmod{256}$ and $c \equiv 240 \bmod{256}$. 


When $t \geq 2$, we again have that $\beta_2 \equiv 3$ or $7$ modulo 8, and hence $\Q_2(\sqrt{\beta_1})$ is a ramified extension of $\Q_2$. On the other hand,
$\Q_2(\sqrt{\alpha_1})$ is $\Q_2$ if $b$ is even and an unramified quadratic extension of $\Q_2$ if $b$ is odd. This occurs, respectively, when $-c = 4^{2t}(1 + 16r)$ and $-c = 4^{2t}(1 + 8(2r+1))$, or replacing $r$ by $-r$, $c = 4^{2t}(-1 + 16r)$ and $c = 4^{2t}(7 + 16r)$.


\textbf{Case 2: $L_c^{I_\B} = \Q(\alpha+\beta)$.}
Let us begin by remarking that $L_c^{I_\B} = \Q(\alpha+\beta)$ implies that $(2)$ ramifies in $\Q(\sqrt{-c})$, since otherwise every prime in $\Q(\sqrt{-c})$ lying over $2$ would occur without ramification, implying that $\Q(\sqrt{-c}) \subset L_c^{I_\B}$, which is false. The same argument applies to $\Q(\sqrt{-(c+1)})$. Hence from \eqref{localquadraticcase} we may write
\begin{equation} \label{initial}
c  = 4^jb_0\qquad 
c+1  = 4^mb_0'
\end{equation}
where $b_0, b_0'$ are integer not equivalent to $0$ or $3$ modulo 4, and $j, m$ are non-negative integers. It follows from Corollary \ref{residueprop} that either $L_c^{D_\B} = L_c^{I_\B}$ or $L_c^{D_\B}$ is a quadratic exention of $\Q$ lying in $L_c^{I_\B}$. But $\Q(\sqrt{c^2 + c})$ is the only quadratic extension of $\Q$ lying in $\Q(\alpha + \beta)$, and hence $\Q(\sqrt{c^2 + c}) \subseteq L_c^{D_\B}$. Hence $2$ must split in $\Q(\sqrt{c^2 + c})$, which is equivalent to $\sqrt{c^2 + c} \in \Q_2$. From \eqref{initial} we have
\begin{equation} \label{blah}
c^2 + c = 4^{j+m}b_0b_0'.
\end{equation}
If one of $b_0$ or $b_0'$ is even, then the other must be odd, because $c$ and $c+1$ are relatively prime. But then from \eqref{blah} $c^2 + c$ has odd 2-adic valuation, and hence cannot be a square in $\Q_2$. 

Therfore $b_0, b_0' \equiv 1 \bmod{4}$, so we may write
\begin{equation} \label{second}
c  = 4^j(1 + 4b_1)\qquad 
c+1  = 4^m(1 + 4b_1'),
\end{equation}
and \eqref{blah} becomes
$$
c^2 + c = 4^{j+m}(1 + 4b_1)(1 + 4b_1') = 4^{j+m}(1 + 4(b_1 + b_1') + 16b_1b_1')).
$$
However, from \eqref{square image} and our knowledge that $\sqrt{c^2 + c} \in \Q_2$, we must have
$1 + 4(b_1 + b_1') + 16b_1b_1') \equiv 1 \bmod{8}$, and hence 
\begin{equation} \label{parity}
\text{$b_1$ and $b_1'$ have the same parity.}
\end{equation}
Let us return to \eqref{second}. Since $c$ and $c+1$ are relatively prime, exactly one of $j$ and $m$ is zero. Suppose that $j = 0$. Then $c + 1 \equiv 2 \bmod{4}$, contradicting \eqref{second}. Therefore $m = 0$, and we can write
\begin{equation} \label{third}
c+1 = 1 + 4(4^{j-1}(1 + 4b_1)),
\end{equation}
so that $b_1' = 4^{j-1}(1 + 4b_1)$. If $j = 1$ then $b_1'$ is odd, so by \eqref{parity} we have that $b_1$ is odd. This is equivalent to $c = 4(5 + 8b)$ for some $b \in \Z$. If $j \geq 2$, then $b_1'$ is even, so from \eqref{parity} we have that $b_1$ is even as well, which is equivalent to $c = 4^{n+2}(1 + 8b)$ for integers $n$ and $b$ with $n \geq 0$. 

We summarize what we have so far: the assumption that 
$L_c^{I_\B} = \Q(\alpha+\beta)$ implies that $(2)$ ramifies in $\Q(\sqrt{-c})$ and $\Q(\sqrt{-(c+1)})$ and splits in $\Q(\sqrt{c^2 + c})$ (though the reverse implication is false). This is equivalent to the assertion that one of the following holds:
\begin{enumerate} 
\item[I.] $c = 4(5 + 8b)$ or
\item[II.] $c = 4^{n+2}(1 + 8b)$,
\end{enumerate}
where $n, b \in \Z$ and $n \geq 0$. We emphasize that for the $c$-values in I and II, there are three possible outcomes: $L_c^{I_\B} = L_c^{D_\B} = \Q(\sqrt{c^2 + c})$; 
$L_c^{I_\B} = \Q(\alpha+\beta)$ and $L_c^{D_\B} = \Q(\sqrt{c^2 + c})$; or
$L_c^{I_\B} = L_c^{D_\B} = \Q(\alpha + \beta)$. These 
correspond to the (4, 1, 2), (2, 2, 2), and (2, 1, 4) cases of Theorem \ref{main1}, respectively. They occur, respectively, if and only if every root of $g(x) = x^4 + 4cx^2 - 4c$ generates a ramified quadratic extension of $Q_2$; one root of $g(x)$ generates an unramified quadratic extension of $Q_2$; and one root of $g(x)$ lies in $Q_2$. This follows from the same reasoning given in the first paragraph of Case 1 on p. \pageref{case1}.
Thus we now determine for which $c$-values in cases I and II we obtain a root of $g(x)$ that either generates an unramified quadratic extension of $\Q_2$ or lies in $\Q_2$. The quadratic formula gives that the roots of $g(x)$ are $-2c \pm 2 \sqrt{c^2 + c}.$

We first consider case I. Note that $$c^2 + c = 4(1 + 8(13 + 41b + 32b^2)),$$
and hence from Proposition \ref{sqrt expansion} we have
\begin{equation*}
\sqrt{c^2+ c} = 2 (1 + 4(13 + 41b) + 8b_3),
\end{equation*}
for some $b_3 \in \Z_2$. 
Let $\alpha_1 = -2c + 2 \sqrt{c^2 + c}$ and $\beta_1 = 
-2c - 2 \sqrt{c^2 + c}$. Then
\begin{align*}
\alpha_1 = -8(5 + 8b) + 4(1 + 4(13 + 41b) + 8b_3)  
& = 4(3 + 4b + 8b_4)\\
\beta_1 = -8(5 + 8b) - 4(1 + 4(13 + 41b) + 8b_3)  
& = 4(-7 - 4b + 8b_5)
\end{align*}
for some $b_4, b_5 \in \Z_2$. From \eqref{localquadraticcase} we have that $\Q_2(\sqrt{\alpha_1})$ is an unramified quadratic extension of $\Q_2$ when $-(3 + 4b) \equiv 3 \bmod{8}$, which is impossible. It is a trivial extension of $\Q_2$ when $-(3 + 4b) \equiv 7 \bmod{8}$, which also impossible. 

On the other hand, $\Q_2(\sqrt{\beta_1})$ is an unramified quadratic extension of $\Q_2$ when $7 + 4b \equiv 3 \bmod{8}$, which occurs if and only if $b$ is odd. This is equivalent to $c = 4(5 + 8(2r + 1))$ for some $r \in \Z$, or in other words $c \equiv 52 \bmod{64}$. Similarly, we have $\Q_2(\sqrt{\beta_1}) = \Q_2$ if and only if 
$7 + 4b \equiv 7 \bmod{8}$, which occurs precisely when $b$ is even, or $c \equiv 20 \bmod{64}$.

We now turn to Case II. This gives
$$c^2 + c = 4^{n+2}(1 + 8b + 4^{n+2} + 4^{n+4}b + 4^{n+6}b^2) = 4^{n+2}(1 + 8(b + 2b_6)),$$
for some $b_6 \in \Z$. Hence from Proposition \ref{sqrt expansion} we have
\begin{equation*}
\sqrt{c^2+ c} = 2^{n+2}(1 + 4(b + 2b_6) - 8b_7),
\end{equation*}
for some $b_7 \in \Z_2$.

Let $\alpha_1 = -2c + 2 \sqrt{c^2 + c}$ and $\beta_1 = 
-2c - 2 \sqrt{c^2 + c}$. Then
\begin{align*}
\alpha_1  = -2^{2n+5}(1 + 8b) + 2^{n+3}(1 + 4(b + 2b_6) - 8b_7)  
& = 2^{n+3}(1 + 4b - 2^{n+2} + 8b_8)\\
\beta_1  = -2^{2n+5}(1 + 8b) - 2^{n+3}(1 + 4(b + 2b_6) - 8b_7)  
& = 2^{n+3}(-1 - 4b - 2^{n+2} + 8b_9)
\end{align*}
for some $b_8, b_9 \in \Z_2$. If $n$ is even, then $v_2(\alpha_1)$ and $v_2(\beta_1)$ are both odd, so both 
$\Q_2(\sqrt{\alpha_1})$ and $\Q_2(\sqrt{\beta_1})$ are ramified, from \eqref{localquadraticcase}. Writing $n = 2\ell + 1$ for $\ell \geq 0$ gives 
\begin{align*}
\alpha_1  = 4^{\ell + 2}(1 + 4b - 2^{2\ell+1} + 8b_8) & = 4^{\ell + 2}(1 + 4b + 8b_{10})\\
\beta_1  = 4^{\ell + 2}(-1 - 4b - 2^{n+2} + 8b_9) & = 4^{\ell + 2}(-1 - 4b + 8b_{11})
\end{align*}
for some $b_{10}, b_{11} \in \Z_2$. 
Thus $\Q_2(\sqrt{\alpha_1})$ is an unramified quadratic extension of $\Q_2$ when $-(1 + 4b) \equiv 3 \bmod{8}$, 
which occurs if and only if $b$ is odd.  This is equivalent to $c = 4^{2\ell + 3} (1 + 8(2r + 1))$ for some $r \in \Z$, or in other words $c = 2^{4k + 6} (16r+9)$. If we require $k \geq 1$ we obtain $c = 2^{4k+2}(16r + 9) = 2^{4k}(64r + 36)$.
Similarly, we have $\Q_2(\sqrt{\alpha_1}) = \Q_2$ if and only if  $-(1 + 4b) \equiv 7 \bmod{8}$, 
which occurs if and only if $b$ is even.  This is equivalent to $c = 4^{2\ell + 3} (1 + 16r))$ for some $r \in \Z$, or in other words $c = 2^{4k + 6} (16r+1)$. If we require $k \geq 1$ we obtain $c = 2^{4k+2}(16r + 1) = 2^{4k}(64r + 4)$.

On the other hand, $\Q_2(\sqrt{\alpha_1})$ is unramified quadratic or trivial only when $1 + 4b \equiv 3, 7 \bmod{8}$, which is impossible.

\textbf{Case 3: $L_c^{I_\B} = \Q(\sqrt{-c}, \sqrt{-(c+1)})$.}  Unlike in Case 2, we have here the convenient fact that $L_c^{I_\B}$ is Galois over $\Q$, implying that every prime in $L_c^{I_\B}$ lying over $2$ is unramified. Moreover, some prime of $L_c^{I_\B}$ lying over 2 occurs without ramification or inertial extension if and only if $2$ splits completely in $L_c^{I_\B}$, which is equivalent to $2$ splitting in both $\Q(\sqrt{-c})$ and $\Q(\sqrt{-(c+1)})$. From \eqref{localquadraticcase} we thus are led to these cases:
\begin{enumerate}
\item[i.] \text{$c = 4^jb$ and $c+1 = 4^mb'$ with $b \equiv b' \equiv 7 \bmod{8}$}
\item[ii.] \text{$c = 4^jb$ and $c+1 = 4^mb'$ with not both $b,b'  \equiv 7 \bmod{8}$},
\end{enumerate}
which are equivalent, respectively, to $L_c^{D_\B} = L_c^{I_\B}$ and thus $(e,f,g) = (2,1,4)$; and $L_c^{D_\B}$ being a quadratic sub-extension of $L_c^{I_\B}$ and thus $(e,f,g) = (2,2,2)$. 

We begin with (i.) above. Note that the relative primality of $c$ and $c+1$ implies that exactly one of $j$ and $m$ is zero. If $m = 0$, then $c + 1 \equiv 7 \bmod{8}$ and $c \equiv 0 \bmod{4}$, which is impossible. Similar arguments rule out the case $m=1$, and so we have $m\ge 2$.  On the other hand, it is easy to check that if $c+1 = 4^m(8r+7)$, then $c= 4^m(8r+7)-1$ satisfies the requirements for any choice of $r$, and so we are done with this case.


Next we address $(ii.)$, which we will break into two considerations.  The first is when $b \equiv 7 \bmod 8$ and $b' \equiv 3 \bmod 8$.  This case is nearly identical to case $(i.)$ in its derivation, so we simply note that this combination implies that $c=4^m(8r+3)-1$ with $m\ge 2$ and $r \in \Z$.  

Next we consider when $b\equiv 3 \bmod 8$ and $b'\equiv 3 \bmod 4$ (encapsulating both the $3 \bmod 8$ and $7 \bmod 8$ cases).  Here, again, we must have $j=0$ to have any chance of satisfying the requirment.  Thus we must have $m\ge 1$.  On the other hand, if $m\ge 2$, then we see immediately that $c \equiv 7 \bmod 8$, a contradiction, so that in fact we have $m=1$.  Letting $c=3+8k$, then, we have $c+1 = 4+8k = 4(2k+1)$, where now, since $m=1$, we have $2k+1=4r+3$.  Thus we have $k=2r+1$ and $c=3+8(2r+1)=16r+11$.  

\section{Higher ramification groups} \label{ramification}

In this section, we calculate the ramification filtration for $c$-values where $L_c$ is totally ramified. We limit ourselves to this case because it is representative of the others, and has the most complicated filtration. It is likely that there is a different ramification filtration for each of the cases in Theorem \ref{main1}, though within each of these cases the filtration is the same. This is what we find in the totally ramified case (see Theorem \ref{totram}). 
For the remainder of this section, we assume that $L_c$ is totally ramified, which by Theorem \ref{main1} occurs if and only if $c\equiv 1 \bmod 4$ or $c=2^{2k+1}\cdot m$ with $k\ge 1$ and $m$ odd.

In this case, we may replace the ground field $\Q$ with $\Q_2$ and obtain an extension (which we again denote $L_c$) that still has degree 8, and Galois group $D_4$. The value group of $L_c$ is $(1/8)\Z$, and thus there is a uniformizing element $\pi \in L_c$ with $v(\pi) = 1/8$. Moreover, we have $L_c = \Q_2(\pi)$. We write $v_\pi$ to denote the $\pi$-adic valuation, so that for instance $v_\pi(2) = 8$.  More generally, $v_\pi(x)=8v_2(x)$.  In particular, because this extension is totally ramified, the $\pi$-adic valuation does not depend on our choice of $\pi$.

Recall that for $i \geq 1$, we have by definition 
$$G_i = \{g \in G : v_\pi(g(\pi) - \pi)) \geq i+1 \},$$
where $G$ is the Galois group of a given local Galois extension (here $L_c/\Q_2$). 
Because $L_c/\Q_2$ is totally ramified, $G_0$ is the inertia group and is isomorphic to $D_4$. 

\subsection{The case $c \equiv 1 \bmod{4}$.}
Our prinicipal task is to find a uniformizing element for $L_c$, i.e. an element $\pi \in L_c$ with $v(\pi) = 1/8$ (and thus necessarily $L_c = \Q_2(\pi)$). The following two paragraphs provide motivation for how this element was found.  

Both $\alpha$ and $\beta$ have power series expansions in our hypothetical $\pi$, so write $\alpha=a_0+a_1\pi+...$ and $\beta=b_0+b_1\pi+...$ with $a_i,b_i\in \{0,1\}$.  Our goal is to produce a combination of $\alpha$ and $\beta$ with odd $\pi$-adic valuation, as then, using $\alpha$s and $\beta$s we can produce something of $\pi$-adic valuation which we can use for $\pi$.  

The fact that $v(\alpha)=v(\beta)=1/4$ from Newton polygon considerations tells us that $a_0=a_1=b_0=b_1=0$ and $a_2=b_2=1$.  
Note then that, because $v_\pi(2)=8$, we have $\alpha+\beta=(a_3+b_3)\pi^3+(a_4+b_4)\pi^4+...$.  On the other hand, from the earlier computation of the minimal polynomial of $\alpha+\beta$, we know that $v(\alpha+\beta)=1/2$.  This tells us that $a_3=b_3$ and that $a_4\neq b_4$.  We also know that $\beta^2=\pi^4+...$, so we know that $v_\pi(\alpha+\beta+\beta^2)\ge 5$.  A SAGE computation reveals that, in fact, $v_pi(\alpha+\beta+\beta^2)=7$, giving an element of odd valuation as desired.

We are thus led to consider
\begin{equation} \label{uniformizer1}
\pi := \frac{\alpha}{2}(\alpha + \beta + \beta^2).
\end{equation}
A computation in SAGE reveals that the minimal polynomial for $\pi$ is the following (here we set $c = 1 + 4m$ with $m \in \Z$): 
\begin{align*}
x^8 & + (16m + 4)x^7 + (64m^3 + 160m^2 + 68m + 8)x^6 \\ & + (768m^4 + 1152m^3 + 560m^2 + 112m + 8)x^5 \\ & + (1536m^6 + 6784m^5 + 7712m^4 + 3832m^3 + 960m^2 + 120m + 6)x^4 \\
 &  + (12288m^7 + 34816m^6 + 35328m^5 + 17408m^4 + 4528m^3 + 600m^2 + 32m)x^3 \\ 
 & + (16384m^9 + 94208m^8 + 180224m^7 + 165376m^6 + 81984m^5 + 21936m^4 + 2552m^3 \\ 
 & \qquad - 104m^2 - 56m - 4)x^2 \\
 &  + (65536m^{10} + 278528m^9 + 483328m^8 + 456704m^7 + 261376m^6 + 94528m^5 \\ 
 & \qquad + 21696m^4 + 3056m^3 + 240m^2 + 8m)x \\ 
 & + (65536m^{12} + 425984m^{11} + 1126400m^{10} + 1671168m^9 + 1587968m^8 + 1035392m^7 \\ 
 & \qquad + 481232m^6 + 162384m^5 + 39828m^4 + 6968m^3 + 828m^2 + 60m + 2)
\end{align*}

This polynomial is Eisenstein, for the constant term is $2 \bmod 4$, and hence has $2$-adic valuation $1$, while every other term is visibly even. Thus, $v(\pi)=1/8$ as claimed.    

Now that we are in possession of a uniformizing element of $L_c$, we compute the $G_i$ directly. 

To aid in the computation of the $G_i$, we now compute $v_\pi(\pi-\sigma(\pi))$ for each of the $7$ non-identity elements $\sigma \in G$.  

\begin{lemma}

$v_\pi(\alpha+\alpha\beta+\beta)=6.$

\end{lemma}

\begin{proof} 

We know from above that $v_\pi(\alpha+\beta+\beta^2)=7$, $v_\pi(\beta)=2$ and $v_\pi(\pm\alpha+\beta)=4$.  Note that $\alpha+\alpha\beta+\beta = (\alpha+\beta+\beta^2) + \beta(\alpha-\beta)$.  The result follows from the strong triangle inequality.

\end{proof}

More generally, using negation and elements of the Galois group, we obtain the following:

\begin{corollary} 

$v_\pi(\pm \alpha \pm \alpha\beta \pm \beta) = 6.$ 

\end{corollary}

Now that we know the valuations of $\pm \alpha, \pm\beta, \pm\alpha \pm \beta, \pm \alpha \pm \alpha\beta \pm \beta^2$, we may readily verify the following using a small amount of algebra and the strong triangle inequality.

$\sigma_1:(\alpha,\beta)\mapsto (-\alpha,\beta)$ : $\pi - \sigma_1(\pi) = \alpha\beta + \alpha \beta^2$, so $v_\pi(\pi-\sigma_1(\pi))=4$

$\sigma_2:(\alpha,\beta)\mapsto (\alpha,-\beta)$ : $\pi - \sigma_2(\pi) = \alpha\beta$, so $v_\pi(\pi-\sigma_2(\pi))=4$

$\sigma_3:(\alpha,\beta)\mapsto (-\alpha,-\beta)$ : $\pi - \sigma_3(\pi) = \alpha\beta^2$, so $v_\pi(\pi-\sigma_3(\pi))=6$

$\sigma_4:(\alpha,\beta)\mapsto (\beta,\alpha)$ : $\pi - \sigma_4(\pi) = (\alpha-\alpha\beta+\beta)(\alpha-\beta)/2$, so $v_\pi(\pi-\sigma_4(\pi))=2$

$\sigma_5:(\alpha,\beta)\mapsto (-\beta,\alpha)$ : $\pi - \sigma_5(\pi) = ((\alpha+\beta)(\alpha+\alpha\beta+\beta)-2\beta^2)/2$, so $v_\pi(\pi-\sigma_5(\pi))=2$

$\sigma_6:(\alpha,\beta)\mapsto (\beta,-\alpha)$ : $\pi - \sigma_6(\pi) = ((\alpha-\beta)(-\alpha-\alpha\beta+\beta)+2\alpha^2)/2$, so $v_\pi(\pi-\sigma_6(\pi))=2$

$\sigma_7:(\alpha,\beta)\mapsto (-\beta,-\alpha)$ : $\pi - \sigma_7(\pi) = (\alpha+\beta)(-\alpha+\alpha\beta+\beta)/2$, so $v_\pi(\pi-\sigma_7(\pi))=2$

Summarizing, we have:

\begin{proposition}

When $c\equiv 1 \bmod 4$, we have the following ramification groups, using the notation above.  

\begin{enumerate}

\item $G_0=G_1=G$
\item $G_2=G_3=\{\sigma_0, \sigma_1, \sigma_2, \sigma_3\}$
\item $G_4=G_5=\{\sigma_0, \sigma_3 \}$
\item $G_n=\{\sigma_0\}$ for $n\ge 6$.  

\end{enumerate}

\end{proposition}

\subsection{The case $c = 2^{2k+1}m$, where $k \geq 1$ and $m$ is odd.}

The computations in this case are very similar to the previous case, so we omit details.  Again we start by producing a uniformizer.  If $k$ is even, write $k=2m$.  Then, one may verify, as above, that $\pi=((\alpha/2^m)^3+((\alpha+\beta)/2^m)+2)/2$ is a uniformizer.  Similarly, if $k=2m+1$, then one may take $\pi=((\alpha/2^m)+((\alpha+\beta)/2^{m+1})^3+2)/2$ as a uniformizer.  

As before we may, for each $\sigma\in G$, compute $v_\pi(\pi-\sigma(\pi))$, yielding the following proposition.  

\begin{proposition}

When $c=2^{2k+1}m$ with $k\ge 1$ and $m$ odd, we have the following ramification groups, using the notation above.  

\begin{enumerate}

\item $G_0=G_1=G$
\item $G_2=G_3=\{\sigma_0, \sigma_3, \sigma_5, \sigma_6\}$
\item $G_4=G_5=G_6=G_7=\{\sigma_0, \sigma_3 \}$
\item $G_n=\{\sigma_0\}$ for $n\ge 8$.  

\end{enumerate}

\end{proposition}

\end{document}